\pdfoutput=1\relax
\documentclass{amsart}

\usepackage{verbatim}
\usepackage[textsize=scriptsize]{todonotes}
\usepackage{tikz-cd}
\usepackage{etoolbox}
\usepackage{etex}
\usepackage[T1]{fontenc}

\usepackage{chemarr}
\usepackage{amssymb}
\usepackage{amsmath}
\usepackage{comment}
\usepackage{mathtools}
\usepackage{rotating}
\usepackage{wrapfig}
\usepackage{outlines}
\usepackage{graphicx}
\usepackage{scalerel}
\usepackage{bbm}
\usepackage{multicol}

\usepackage{amsthm}
\usepackage{chngcntr}

\usepackage{hyperref}
\usepackage{cleveref}

\hypersetup{
   colorlinks,
   linkcolor={red},
   citecolor={green!30!black},
   urlcolor={blue}
}

\usepackage{tikz}
\usetikzlibrary{matrix,arrows,decorations}
\usepackage{tikz-cd}

\usepackage{adjustbox}
\let\oldtocsection=\tocsection
\let\oldtocsubsection=\tocsubsection
\let\oldtocsubsubsection=\tocsubsubsection
\renewcommand{\tocsection}[2]{\hspace{0em}\oldtocsection{#1}{#2}}
\renewcommand{\tocsubsection}[2]{\hspace{1em}\oldtocsubsection{#1}{#2}}
\renewcommand{\tocsubsubsection}[2]{\hspace{2em}\oldtocsubsubsection{#1}{#2}}

\theoremstyle{definition}

\newtheorem{nul}{}[section]
\newtheorem{dfn}[nul]{Definition}

\newtheorem{rmk}[nul]{Remark}

\newtheorem{cnstr}[nul]{Construction}

\newtheorem{exm}[nul]{Example}

\newtheorem{qst}[nul]{Question}

\newtheorem*{dfn*}{Definition}
\newtheorem*{axm*}{Axiom}
\newtheorem*{ntn*}{Notation}
\newtheorem*{exm*}{Example}
\newtheorem*{exr*}{Exercise}
\newtheorem*{int*}{Intuition}
\newtheorem*{qst*}{Question}
\newtheorem*{rmk*}{Remark}

\theoremstyle{plain}

\newtheorem{thm}[nul]{Theorem}
\newtheorem{prop}[nul]{Proposition}

\newtheorem{lem}[nul]{Lemma}

\newtheorem{cor}[nul]{Corollary}

\newtheorem*{thm*}{Theorem}
\newtheorem*{prop*}{Proposition}
\newtheorem*{cor*}{Corollary}
\newtheorem*{lem*}{Lemma}
\newtheorem*{cnj*}{Conjecture}

\DeclareMathOperator*{\colim}{colim}

\DeclareMathOperator{\cof}{cof}

\DeclareMathOperator{\coker}{coker}

\DeclareMathOperator{\THH}{THH}

\DeclareMathOperator{\Map}{Map}
\DeclareMathOperator{\End}{End}

\DeclareMathOperator{\Fun}{Fun}
\DeclareMathOperator{\map}{map}

\DeclareMathOperator{\perf}{perf}

\DeclareMathOperator{\Spec}{\mathrm{Spec}}

\DeclareMathOperator{\im}{im}

\DeclareMathOperator{\Mod}{Mod}

\def\Ind{\mathrm{Ind}}

\def\uloc{\mathcal{U}_{\mathrm{loc}}}

\def\A{\mathbb{A}}

\def\E{\mathbb{E}}
\def\F{\mathbb{F}}

\def\Q{\mathbb{Q}}
\def\R{\mathbb{R}}
\def\Ss{\mathbb{S}}
\def\Z{\mathbb{Z}}

\def\CC{\mathcal{C}}
\def\DD{\mathcal{D}}
\def\AA{\mathcal{A}}

\def\o{\mathbbm{1}}

\def\nil{\mathrm{nil}}

\def\Alg{\mathrm{Alg}}
\def\Cat{\mathrm{Cat}}

\def\Spaces{\mathcal{S}}

\def\QCoh{\mathrm{QCoh}}

\def\Rep{\mathrm{Rep}}

\def\op{\mathrm{op}}
\def\heart{\heartsuit}
\def\perf{\mathrm{perf}}
\def\nc{\mathrm{nc}}

\def\Id{\mathrm{Id}}

\def\BB{\mathcal{B}}

\newcommand{\pushout}{\arrow[ul, phantom, "\ulcorner", very near start]}
\newcommand{\pullback}{\arrow[dr, phantom, "\lrcorner", very near start]}

\newcommand\xqed[1]{%
  \leavevmode\unskip\penalty9999 \hbox{}\nobreak\hfill
  \quad\hbox{#1}}
\newcommand\tqed{\xqed{$\triangleleft$}}

\usepackage[margin=1.5in]{geometry}

\newtoggle{draft}
\togglefalse{draft}

\iftoggle{draft} {
  \newcommand{\NB}[1]{\todo[color=gray!40]{#1}}
  \newcommand{\TODO}[1]{\todo[color=red]{#1}}
}{ 
  \newcommand{\NB}[1]{}
  \newcommand{\TODO}[1]{}
  \renewcommand{\todo}[1]{}
  \renewcommand{\todo}[1]{}
}

\title{On the \texorpdfstring{$K$}{K}-theory of regular coconnective rings}
\date{\today}

\author{Robert Burklund}
\address{Department of Mathematics, MIT, Cambridge, MA, USA}
\email{burklund@mit.edu}

\author{Ishan Levy}
\address{Department of Mathematics, MIT, Cambridge, MA, USA}
\email{ishanl@mit.edu}
\thanks{The second author is supported by the NSF Graduate Research Fellowship under Grant No. 1745302.}

\begin{document}

\begin{abstract}
We show that for a coconnective ring spectrum satisfying regularity and flatness assumptions, its algebraic $K$-theory agrees with that of its $\pi_0$. We prove this as a consequence of a more general devissage result for stable infinity categories. Applications of our result include giving general conditions under which $K$-theory preserves pushouts, generalizations of $\A^n$-invariance of $K$-theory, and an understanding of the $K$-theory of categories of unipotent local systems.

\end{abstract}

\maketitle

\setcounter{tocdepth}{1}
\tableofcontents
\vbadness 5000


\section{Introduction}
\label{sec:intro}
In this paper, we examine the relationship between coconnectivity, regularity, and algebraic $K$-theory. We identify the $K$-theory of a large collection of coconnective rings with that of their $\pi_0$.

\begin{thm} \label{thm:opener}
  Given a coconnective $\E_1$-algebra $R$ such that
  \begin{enumerate}
  \item $\pi_0R$ is left regular coherent and
  \item $\tau_{\leq-1}R$ has tor amplitude in $[-\infty,-1]$ as a right $\pi_0R$ module,
  \end{enumerate}
  the natural map in connective $K$-theory
  \[ K(\pi_0R) \to K(R) \]
  is an equivalence and both $\pi_0R$ and $R$ have vanishing $K_{-1}$.
\end{thm}

Although not immediately clear, \Cref{thm:opener} is a devissage theorem.
The core step in the proof is an application of Quillen's devissage theorem \cite[Theorem 4]{quillenhigherktheory} and condition (1) is exactly what is needed for the category of perfect $\pi_0R$-modules to have a $t$-structure with heart finitely presented $\pi_0R$-modules. The essential novelty in \Cref{thm:opener} comes from condition (2) as a simple condition, easily checked in practice\footnote{Condition $(2)$ is satisfied if $\pi_{-i}R$ has tor dimension $<-i$ as a right $\pi_0R$-module.}, which allows us to conclude. As a demonstration we work through the prototypical example of devissage.


\begin{exm}
  From the localization sequence
  \[ \Mod(\Z)^{p\mathrm{-nil}} \hookrightarrow \Mod(\Z) \twoheadrightarrow \Mod(\Z[1/p]) \]
  we obtain a cofiber sequence of noncommutative motives 
  \[ \uloc\left( \Mod(\Z)^{p\mathrm{-nil}} \right) \to \uloc(\Mod(\Z)) \to \uloc(\Mod(\Z[1/p])). \]
  Identifying $\F_p$ as a generator of $\Mod(\Z)^{p\mathrm{-nil}}$ we have an identification
  \[ \Mod(\Z)^{p\mathrm{-nil}} \simeq \Mod( \End_{\Z}(\F_p) ). \]
  Devissage can then be phrased as the assertion that $K(\End_{\Z}(\F_p)) \simeq K(\F_p)$, from which we obtain a cofiber sequence on algebraic $K$-theory
  \[ K( \F_p ) \to K(\Z) \to K(\Z[1/p]). \]
  In order to prove this using \Cref{thm:opener}
  we compute the homotopy groups of $\End_{\Z}(\F_p)$, which are
  \[ \pi_s\End_{\Z}(\F_p) \cong \begin{cases} \F_p & s=0,-1 \\  0 & \text{otherwise} \end{cases} \]
  and observe that conditions (1) and (2) are satisfied.
  \tqed
\end{exm}

In this example we entirely avoided descending into the category of modules to check the existence of filtrations, instead operating at the level of categories, motives and rings throughout.
This change represents a considerably gain in practical usability.
Before moving on, let us point out that this example also highlights another key feature of devissage which is sometimes overlooked.
While localization sequences occur at the level of noncommutative motives, devissage is specific to $K$-theory\footnote{For example $\THH(\End_{\Z}(\F_p))$ is the fiber of the map $\THH(\Z) \to \THH(\Z[1/p])$, which doesn't agree with $\THH(\F_p)$.}.


We prove \Cref{thm:opener} as a corollary of our main result \Cref{thm:cat-main}, which is a more general devissage result
taking place at the level of stable categories:

\begin{thm}\label{thm:cat-main}
  Let $ F : \CC \to \DD $ be an exact functor between small, stable, idempotent complete categories and let $(\CC_{\geq 0}, \CC_{\leq 0})$ be a bounded $t$-structure on $\CC$.
  If we assume that 
  \begin{enumerate}
  \item[(A)] the image of $F$ generates $\DD$ and
  \item[(B)] $F$ is fully faithful when restricted to $\CC^\heart$
  \end{enumerate}
  then there is a corresponding bounded $t$-structure on $\DD$ for which $F$ is $t$-exact.
  Moreover, the induced maps on connective $K$-theory
  \begin{center}
    \begin{tikzcd}
      K(\CC^\heart) \ar[r] \ar[d] & K(\DD^\heart) \ar[d] \\
      K(\CC) \ar[r] & K(\DD)
    \end{tikzcd}
  \end{center}
  are all equivalences and both $\CC$ and $\DD$ have vanishing $K_{-1}$.
\end{thm}


We prove \Cref{thm:cat-main} in \Cref{sec:main} where the key step is constructing a bounded $t$-structure on $\DD$ for which $F$ is $t$-exact.
This is the most technical point in the proof and it uses all of the conditions of the theorem in an essential way.
In fact, as a byproduct of this argument we obtain relatively fine-grained control over the abelian category $\DD^{\heart}$. Specifically, $\CC^\heart$ sits inside $\DD^{\heart}$ as a full subcategory and every object of $\DD^{\heart}$ has a finite length filtration with associated graded in $\CC^{\heart}$.
At this point we apply Barwick's theorem of the heart \cite{Barwick_2015} to identify the $K$-theory of $\CC$ with that of $\CC^\heart$ and the $K$-theory of $\DD$ with that of $\DD^\heart$. The proof ends by applying Quillen's devissage \cite[Theorem 4]{quillenhigherktheory} to the inclusion of abelian categories $\CC^\heart \to \DD^\heart$.

Examining the relation between $\CC^\heart$ and $\DD^\heart$  we see that conditions (A) and (B) together can be thought of as asking that $\DD$ behave like a category of unipotent local systems with coefficients in $\CC$.
Indeed, (B) is analogous to the fact that maps between trivial representations can be computed on underlying and (A) is analogous to the fact that unipotent representations are generated from trivial representations under extensions.
For this reason we say that a map is \emph{unipotent} if it satisfies these conditions.
With this reformulation we can now introduce the key slogan of this paper:
\vspace{0.1cm}
\begin{quote}
  \emph{Devissage is the invariance of $K$-theory under unipotent maps.}
\end{quote}
\vspace{0.1cm} 

In \Cref{sec:complements} we deduce \Cref{thm:opener} from \Cref{thm:cat-main} and discuss several points which are complementary to \Cref{thm:cat-main}. A subtlety worth noting here is that up to this point we have been working entirely with \textit{connective} $K$-theory as this is the setting where Barwick's theorem of the heart and Quillen's devissage are applicable. Since we are not aware of any examples where these results fail in negative $K$-theory we are led to ask:

\begin{qst} \label{qst:intro-neg-heart}
  Do the theorem of the heart and devissage hold for negative $K$-theory?
\end{qst}

In Sections \ref{sec:applications} and \ref{sec:examples} we turn to applications of our main theorem.
Combining our work with the work of Land--Tamme on the $K$-theory of pullbacks we provide general conditions under which $K$-theory preserves pushouts. 

\begin{thm}\label{thm:discrete-pushouts}
  Suppose $C \xleftarrow{g} A \xrightarrow{f} B$ is a span of discrete rings
  where $A$ is left regular coherent and both $f$ and $g$ are right faithfully flat.
  Then connective $K$-theory preserves the pushout of this span.
\end{thm}

This generalizes Waldhausen's results about the $K$-theory of generalized free products \cite{Waldhausen-free}. Using similar techniques, we then obtain an $\A^n$-invariance result for $K$-theory.
 
\begin{thm}[$\A^n$-invariance of algebraic $K$-theory]\label{thm:introan-inv}
  Let $\CC$ be a small, stable, idempotent complete category equipped with a bounded $t$-structure.
  Then $K_i(\CC) \cong K_i(\CC[x_1,\dots,x_n])$ for $i \geq n-1$.
\end{thm}

In the $n=1$ case, for a regular, Noetherian ring this recovers Quillen's fundamental theorem of algebraic $K$-theory \cite[Theorem 8]{quillenhigherktheory}. An alternative proof extending the result to regular coherent rings was given by Waldhausen, again in the $n=1$ case \cite{Waldhausen-free}.
The case of non-Noetherian regular coherent rings and $n>1$ is more difficult, because $R[x]$ may not even be coherent, so one cannot induct on $n$ in the obvious way. 
Nevertheless, the $n>1$ case was already known for example when $R$ is a discrete ring by the Ferrell-Jones conjecture for the groups $\Z^n$. See for example \cite[Corollary 2]{davis2008remarks}, which along with the $n=1$ case of \Cref{thm:introan-inv} implies the general case.
The degree bounds in the above theorem ultimately come from our use of connective $K$-theory and a positive answer to \Cref{qst:intro-neg-heart} would allow us to remove these restrictions.

In \Cref{subsec:locsys} we examine the category of unipotent local systems on a connected space $X$ with coefficients in a category $\CC$ with a bounded $t$-structure.
As one might expect, we find that the $K$-theory of unipotent local systems agrees with the $K$-theory of $\CC$, generalizing \cite[Theorem 4.8]{antieau2018ktheoretic}.
In the final pair of subsections we work through a collection of examples which demonstrate that the conditions in \Cref{thm:opener} cannot be weakened.


\subsection*{Notations and Conventions}\ 

In order to preserve the brevity of this paper we assume the reader is generally familiar with higher algebra and algebraic $K$-theory.
We also make use of the following notations and conventions throughout.

\begin{itemize}
\item The term category will refer to an $\infty$-category as developed by Joyal and Lurie.
\item $\Map(a,b)$ will denote the space of maps from $a$ to $b$ (in some ambient category).
\item In a stable category $\map(a,b)$ will denotes the mapping spectrum between $a$ and $b$.
\item For an $\E_1$-algebra $R$, $\Mod(R)$ will refer to its category of left modules.
\item We use $\CC, \DD$ to denote small, idempotent complete stable categories, and use $\Cat^{\perf}$ to denote the category of such categories and exact functors.
\item Given an exact functor $F:\CC \to \DD$ in $\Cat^{\perf}$, $F^*:\Ind(\CC) \to \Ind(\DD)$ denotes $\Ind(F)$, and $F_*: \Ind(\DD) \to \Ind(\CC)$ denotes the right adjoint of $F^*$.
\item We use $\uloc$ for the noncommutative motive (or just nc motive for short) functor of Blumberg--Gepner--Tabuada \cite{BGT}.
\item We use $K(-)$ for connective $K$-theory and $K^{\mathrm{nc}}(-)$ for non-connective $K$-theory.
\item We use $x_n$ for a polynomial generator in degree $n$ and $\epsilon_n$ for an exterior generator in degree $n$. As an example, $\Ss[x_n]$ is the free $\E_1$-algebra on a class in degree $n$.
\item We use $\CC[x_n]$ as notation for $\CC \otimes \Mod(\Ss[x_n])$ and similarly for exterior generators.
\end{itemize}

\subsection*{Acknowledgments}\

We would like to thank Andrew Blumberg, Jeremy Hahn, Mike Hopkins, Haynes Miller, Piotr Pstragowski, and Lucy Yang for helpful conversations related to this work.
We would also like to thank Andrew Blumberg for helpful comments on a draft of this paper.

Our deepest thanks go to Markus Land and Georg Tamme for many discussions about this work and its applications. It was they who first asked us when the connective cover map of a coconnective $\E_1$-algebra induces a $K$-equivalence, which was the origin of this work.


\section{The Main Theorem}
\label{sec:main}

In this section we prove our main theorem.

\begin{thm}
  Let $ F : \CC \to \DD $ be an exact functor between small, stable, idempotent complete categories and let $(\CC_{\geq 0}, \CC_{\leq 0})$ be a bounded $t$-structure on $\CC$.
  If we assume that 
  \begin{enumerate}
  \item[(A)] the image of $F$ generates $\DD$ and
  \item[(B)] $F$ is fully faithful when restricted to $\CC^\heart$
  \end{enumerate}
  then there is a corresponding bounded $t$-structure on $\DD$ for which $F$ is $t$-exact.
  Moreover, the induced maps on connective $K$-theory
  \begin{center}
    \begin{tikzcd}
      K(\CC^\heart) \ar[r] \ar[d] & K(\DD^\heart) \ar[d] \\
      K(\CC) \ar[r] & K(\DD)
    \end{tikzcd}
  \end{center}
  are all equivalences and both $\CC$ and $\DD$ have vanishing $K_{-1}$.
\end{thm}

Before proceeding, we give a sketch of the strategy we follow in proving this theorem.
The final step is applying Barwick's theorem of the heart and Quillen's devissage theorem to produce $K$-theory equivalences. In order to apply these results we need to produce a bounded $t$-structure on $\DD$ which is relatively well behaved. The key idea is that after passing to categories of ind-objects it is in fact quite easy to produce such a $t$-structure. Condition (B) is then rigged so that we have the control necessary to restrict this $t$-structure to compact objects in $\Ind(\DD)$ (i.e. $\DD$).
For the remainder of this section the notation from the statement of \Cref{thm:cat-main} will remain in place and we assume $F$ satisfies conditions (A) and (B).

Passing to ind-completions gives us an induced diagram
\begin{center}
  \begin{tikzcd}
    \CC \ar[r,"F"]\ar[d,hook] & \DD\ar[d,hook]\\
    \ar[r,"F^*"]\Ind(\CC) & \Ind(\DD)
  \end{tikzcd}
\end{center}
where the vertical arrows are each the inclusion of the full subcategory of compact objects\footnote{We will suppress any further mention of these inclusions.}.

As promised, we begin by producing a $t$-structure on the level of ind-objects.
This is rather easy since the category of ind-objects is presentable.
\begin{lem}[{\cite[Proposition 1.4.4.11]{HA}}]\label{lem:generalbigt}
  Let $\AA$ be a presentable, stable category.
  If $\{X_\alpha\}$ is a small collection of objects in $\AA$, then there is an accessible $t$-structure, $(\AA_{\geq 0}, \AA_{\leq 0})$, on $\AA$ such that $\AA_{\geq 0}$ is the smallest full subcategory of $\AA$ containing each $X_\alpha$ and closed under colimits and extensions. The full subcategory of coconnective objects is characterized by the condition $Y \in \AA_{\leq 0}$ if and only if $\Map(\Sigma X_{\alpha},Y) = 0$ for each $X_\alpha$.
\end{lem}

We equip $\Ind(\CC)$ with the $t$-structure whose connective part is generated by $\CC_{\geq 0}$ and we equip $\Ind(\DD)$ with the $t$-structure whose connective part is generated by $F(\CC_{\geq 0})$.

\begin{lem} \label{lem:big-t}
  $F^*$ is $t$-exact.
\end{lem}

\begin{proof}  
  $F^*$ sends connective objects to connective objects by construction.  
  To show that $F^*$ preserves coconnectivity we need to check that
  for every $c \in \CC_{\geq 1}$ and $x \in \Ind(\CC)_{\leq 0}$
  the mapping space $\Map(F^*(c),F^*(x))$ is contractible.  
  Since the $t$-structure on $\Ind(\CC)$ restricts to compact objects
  we can write $x$ as a filtered colimit of compact, coconnective objects.    
  This implies (since $F^*$ is a left adjoint)
  that it suffices to prove $\Map(F^*(c), F^*(x)) = 0$ when $x$ is compact.
  Via the boundedness of the $t$-structure on $\CC$ this follows from condition (B).  
  
\end{proof}

At this point are now ready to prove that the $t$-structure on $\Ind(\DD)$ restricts to a bounded $t$-structure on $\DD$. The main idea in proving this is that on the one hand, (A) guarantees that every object in $\DD$ is only finitely many steps away from being in the image of $F$, while on the other hand, the $t$-structure on $\CC$ can be used to produce a rich collection of compact objects in $\Ind(\DD)^\heart$.

\begin{prop} \label{prop:small-t}
  The $t$-structure on $\Ind(\DD)$ restricts to a bounded $t$-structure on $\DD$.
  Each $d \in \DD^\heart$ has a finite filtration with associated graded in the image of $F|_{\CC^{\heart}}: \CC^\heart \to \DD^\heart$.
\end{prop}

\begin{proof}
 Consider the subcategory of $\DD$ of objects $d$ which satisfies the following condition,
  \begin{enumerate}
  \item[$(*)$] $d$ is bounded and each $\pi_i^\heart(d)$ has a finite filtration whose associated graded lies in $\CC^\heart \subseteq \Ind(\DD)^\heart$.
  \end{enumerate}
  Note that if $(*)$ is satisfied, then because compact objects are closed under extensions and the image of $\CC^{\heart}$ is compact, each $\pi_i^{\heart}(d)$ as well as $d$ itself is compact. In order to prove the proposition it will suffice to show that the every object in $\DD$ satisfies $(*)$. We begin by observing that for $c \in \CC^\heart$ its image under $F$ satisfies $(*)$. Using hypothesis (A) it will now suffice to show that full subcategory of objects of $\DD$ satisfying $(*)$ is thick. 

  Now, since the condition $(*)$ is stated entirely in terms of homotopy groups, it will suffice to show that the corresponding condition $(**)$ on the level of the heart cuts out a subcategory closed under kernels, cokernels and extensions\footnote{This uses that fact that these operations suffice to describe how homotopy groups change under cofiber sequences and idempotents}.
  \begin{enumerate}
  \item[$(**)$] $d \in \Ind(\DD)^{\heart}$ has a finite filtration whose associated graded lies in $\CC^\heart$.
  \end{enumerate}
  
  Since filtrations can be pasted, the collection of objects satisfying $(**)$ is closed under extensions. We will handle kernels and cokernels simultaneously. Suppose $A, B \in \Ind(\DD)^\heart$ satisfy $(**)$ and $r$ is a map between them. We can paste the filtrations on $A$ and $B$ to form a filtration on the cofiber, $\cof(r)$, whose associated spectral sequence converges to (an associated graded of) $\pi_*^\heart\cof(x)$ and has $E_1$-page given by the associated graded of $A$ on the $0$-line and the associated graded of $B$ on the $1$-line. 
  By hypothesis, the $E_1$-page of this spectral sequence involves only objects of $\CC^\heart$.
  Now, since $F$ is $t$-exact and fully faithful on $\CC^{\heartsuit}$, kernels and cokernels of maps between objects in the image of $F|_{C^{\heartsuit}}$ remain in the image of $F|_{C^{\heartsuit}}$.
  Consequently, as we run the differentials in this spectral sequence we will not leave $\CC^\heart$.
  Since the spectral sequence has only finitely many pages we learn that it abuts to a filtration of the desired type on the kernel and cokernel of $r$ (which appears as $\pi_1^\heart(\cof(r))$ and $\pi_0^\heart(\cof(r))$ respectively).
\end{proof}

We now recall Quillen's devissage and the theorem of the heart, which we use to finish the proof of the main theorem.

\begin{thm}[{\cite{Barwick_2015}} Barwick's theorem of the heart]\label{thm:ofheart} Let $\CC$ be a stable category with bounded $t$-structure. Then the inclusion $\CC^{\heart} \to \CC$ induces an equivalence on connective $K$-theory.
\end{thm}

\begin{thm}[{\cite[Theorem 4]{quillenhigherktheory}} Quillen's devissage]\label{thm:quilldevissage} Let $\AA \subset \BB$ be an exact fully faithful inclusion of abelian categories with $\AA$ closed in $\BB$ under subobjects, and such that every object of $\BB$ has a finite filtration with associated graded in $\AA$. Then the inclusion $\AA \to \BB$ induces an equivalence on connective $K$-theory.
\end{thm}

\begin{proof}[Proof (of \Cref{thm:cat-main}).]
  At this point we have already shown that $\DD$ admits a bounded $t$-structure (\Cref{prop:small-t}) for which $F$ is $t$-exact (\Cref{lem:big-t}). This allows us to examine the square on $K$-theory,
    \begin{center}
    \begin{tikzcd}
      K(\CC^\heart) \ar[r] \ar[d] & K(\DD^\heart) \ar[d] \\
      K(\CC) \ar[r,"K(F)"] & K(\DD).
    \end{tikzcd}
  \end{center}
  Because the $t$-structures on $\CC$ and $\DD$ are bounded, we can use \Cref{thm:ofheart} to see that
  the vertical maps are equivalences.

  In order to finish the proof it suffices to show that the top horizontal map is an equivalence, which we show by applying \Cref{thm:quilldevissage}.
  The map $f: \CC^\heart \to \DD^\heart$ is fully faithful and exact by construction, and we showed that the filtration condition is satisfied in \Cref{prop:small-t}.

  It remains to check that if $d \in \DD^\heart$ is a subobject of $c \in \CC^\heart$, then $d \in \CC^\heart$. Using the exactness of the inclusion it will suffice to instead show that $\coker(d \to c) \in \CC^\heart$. Using \Cref{prop:small-t} we can equip $d$ with a finite filtration with associated graded in $\CC^\heart$. The cokernel $\coker(d \to c)$ can be produced by successively quotienting $c$ by the pieces in the associated graded of the filtration on $d$, thus we only need to know that quotients by subobjects coming from $\CC^\heart$ stay in $\CC^\heart$. This last statement follows from the fact that $f$ is fully faithful and exact.
  
\end{proof}

\begin{rmk}\label{rmk:optimality}
  We end this section by observing that the following converse to \Cref{thm:cat-main} holds: if we have $F: \CC \to \DD$ a map which we can use the theorem of the heart and Quillen's devissage to prove is a $K$-equivalence, then (A) and (B) must hold.

  To see this, first note that to apply the theorem of the heart we need bounded $t$-structures on $\CC$ and $\DD$ so that $F$ is $t$-exact.
  To apply Quillen's devissage, the induced functor on hearts should by fully faithful and its image should generate $\DD^\heart$ under extensions.
  The condition on generating $\DD^\heart$ implies (A).
  The $t$-exactness of $F$, plus fully faithfulness on the heart implies (B).

  Simply put, this is saying that \Cref{thm:cat-main} is essentially equivalent to the combination of Quillen's devissage and the theorem of the heart.
\tqed
\end{rmk}


\section{Complements}
\label{sec:complements}

In this section we discuss a couple points which are complementary to \Cref{thm:cat-main}.
We begin by introducing some ideas from noncommutative geometry which provide a convenient language for thinking about our main theorem.
Then, we discuss variants of the main theorem and prove \Cref{thm:opener} from the introduction.
In the third subsection we briefly consider the simplifications and extensions to \Cref{thm:cat-main} we can make when $\CC^\heart$ is Noetherian.
We end the section by briefly discussing negative $K$-groups.

\subsection{Some nc geometry}\ 

For us noncommutative geometry refers to thinking about small idempotent complete stable categories equipped with a ``positive half'' closed under finite colimits and extensions. This is quite close to established notions of noncommutative geometry such as in \cite{orlovsmprop2016}, with the notable difference being that we work relative to the sphere rather than relative to a discrete base ring $k$. We explore this setting in some depth in \cite{ncgstuff} and in this section we build on the groundwork from that paper\footnote{Even though we cite results in \cite{ncgstuff}, we do not use anything particularly difficult from there, and so the results here can be considered independent of that paper.}.
Before proceeding we remind the reader of the main definitions.

\begin{dfn}
  We use $\Cat^{\perf}$ to denote the category of small idempotent complete stable categories.
  Our main objects of study are objects of $\Cat_{\geq0}^{\perf}$.
  This is the category of $\CC \in \Cat^{\perf}$ equipped with an idempotent complete prestable\footnote{As introduced in \cite[Appendix C]{SAG}} full subcategory $\CC_{\geq0}$ that generates it. Being prestable amounts to asking that $\CC_{\geq0}$ be closed under finite colimits and extensions.
  Often, we abuse notation by writing $\CC \in \Cat_{\geq0}^{\perf}$, leaving the subcategory of positive objects, $\CC_{\geq0}$, implicit.
\end{dfn}

\begin{exm}\label{exm:ringprestable}
  Given an $\E_1$-algebra $R$, the category compact $R$-modules, $\Mod(R)^{\omega}$, naturally lives in $\Cat^{\perf}_{\geq0}$. The positive objects are those built from $R$ via extensions, finite colimits and retracts.
  \tqed
\end{exm}

Given $\CC \in \Cat^{\perf}_{\geq0}$,
the subcatgory $\Ind(\CC_{\geq0}) \subset \Ind(\CC)$ 
determines a $t$-structure on $\Ind(\CC)$ (see \Cref{lem:generalbigt}). In fact, $\CC_{\geq0}$ can be recovered from the data of this $t$-structure.

\begin{exm}
  In the $t$-structure associated to \Cref{exm:ringprestable}, a connective object is one built out of copies of $R$ under colimits and extensions, and a coconnective object is one whose underlying spectrum is coconnective.
  \tqed
\end{exm}

\begin{dfn} \label{dfn:ncg-defs}
  Given $\CC \in \Cat_{\geq0}^{\perf}$, 
  \begin{itemize}  
  \item $\CC$ is \textit{regular} if the $t$-structure on $\Ind(\CC)$ restricts to $\CC$,
  \item $\CC$ is \emph{bounded} if each $c \in \CC$ is bounded as an object of $\Ind(\CC)$,
  \item a functor $\CC \to \DD$ is \textit{quasi-affine} if its image generates $\DD$ under finite colimits and retracts and 
  \item a quasi-affine functor $\CC \to \DD$ is \textit{unipotent} if it is fully faithful on $\Ind(\CC)^\heart$. \tqed
  \end{itemize}
\end{dfn}

\begin{exm} \label{lem:discreteregular}
  If $R$ is a discrete ring then, as a result of well-known arguments, $\Mod(R)^{\omega}$ is regular
  iff $R$ is left regular coherent.
  For a proof that states things in this way see \cite[Proposition 2.4]{ncgstuff}.
  \tqed
\end{exm}

\begin{exm}\label{prop:polyringreg}
  We show in \cite[Proposition 2.16]{ncgstuff} that if $\CC$ is regular, and $n\neq0$, then $\CC[x_n]$ is regular.\tqed
\end{exm}

\begin{rmk}
  If we think in terms of categories of quasicoherent sheaves,
  the reasoning behind the term quasi-affine is relatively transparent.
  
  The term unipotent bears more explanation.
  The key identifying features of a unipotent group are that maps between trivial representations can be computed on underlying and every representation is built out of extensions of trivial reps.
  Our definition takes these properties as the definition of unipotent.
  
  Note that conditions (A) and (B) of \Cref{thm:cat-main} are equivalent to saying that $F$ is unipotent. This lets us reinterpret \Cref{thm:cat-main} as saying that regularity can be transferred along unipotent maps.
  \tqed
\end{rmk}

\subsection{Other forms of the main theorem}\

In practice unipotence can be difficult to check
so we recall an equivalent condition which is often more transparent. The functor $F^*:\Ind(C) \to \Ind(D)$ has a right adjoint $F_*$, which is colimit preserving since $F^*$ preserves compact objects.

\begin{lem}[{\cite[Corollary 4.12]{ncgstuff}}] \label{lem:condbprime} 
  For a map $F : \CC \to \DD$ as in \Cref{thm:cat-main} condition (B)
  is equivalent to: 
  \begin{enumerate}
  \item[(B$'$)] For every $c \in \CC^{\heart}$,
    the cofiber of the unit map $c \to F_*F^*(c)$ is $\leq -1$ in the $t$-structure on $\Ind(\CC)$.
  \end{enumerate}
\end{lem}
\begin{proof}[Proof sketch]
  Unraveling (B$'$) gives the statement that the cofiber of $\map(d,c) \to \map(Fd,Fc)$ is coconnected for all $c \in \CC^{\heart}$ and $d \in \CC_{\geq0}$.
  This visibly implies (B).  
  The key point in proving the reverse implication is using the fact that $\CC^{\heart}$ is closed under extensions\footnote{In this paper we only use that (B$'$) implies (B) and not the reverse implication.}. 
\end{proof}

We now provide a version of \Cref{thm:cat-main} for categories of modules over an $\E_1$-algebra from which \Cref{thm:opener} will follow.

\begin{prop}\label{prop:relativeopener}
  Let $f: A \to B$ be a map of $\E_1$-algebras such that	
  \begin{enumerate}
  \item $\Mod(A)^{\omega}$ is bounded and regular and
  \item $\cof(f)$ has tor amplitude in $[-\infty,-1]$ as a right $A$-module,
  \end{enumerate}
  then $\Mod(B)^{\omega}$ is bounded and regular, the base-change functor $(-)\otimes_AB$ is $t$-exact 
  and the map $K(A) \to K(B)$ is an equivalence.
\end{prop}

\begin{proof}
  We apply \Cref{thm:cat-main} to the base-change functor
  \[ B \otimes_{A} - : \Mod(A)^{\omega} \to \Mod(B)^{\omega}. \]  
  This functor is quasi-affine since $A$ is sent to $B$ which is a generator.    
  Condition (B$'$) of \Cref{lem:condbprime} asks that for every $N \in \Mod_{A}^{\heart}$ the $A$-module $\cof(f)\otimes_{A}N$ be coconnected.
  This is equivalent to the given tor amplitude bound on $\cof(f)$.
\end{proof}

\begin{proof}[Proof (of \Cref{thm:opener}).]
  We apply \Cref{prop:relativeopener} to the connective cover map $f:\pi_0R \to R$.
  From \Cref{lem:discreteregular} we know that $\pi_0R$ is left regular coherent iff
  $\Mod(\pi_0R)^{\omega}$ is regular.
  Boundedness is automatic.
  Since $\tau_{<0}R \simeq \cof(f)$,
  the tor amplitude bounds in \Cref{prop:relativeopener} and \Cref{thm:opener} match up.
\end{proof}

\begin{rmk}\label{rmk:ringoptimality}
  Amplifying \Cref{rmk:optimality}, we note that the conditions of \Cref{thm:opener} are actually \emph{equivalent} to (A) and (B$'$) for the connective cover map, implying a converse to this theorem.
  \tqed
\end{rmk}

\begin{rmk}
  The reader might wonder when the the abelian categories $\CC^\heartsuit$ and $\DD^\heartsuit$ appearing in \Cref{thm:cat-main} are equivalent. We remark that this will be the case as soon as $F$ satisfies:
  \begin{itemize}
  \item[(C)] for every $c \in \CC^{\heart}$,
    the cofiber of the unit map $c \to F_*F^*(c)$ is $\leq -2$ in the $t$-structure on $\Ind(\CC)$\footnote{See \cite[Proposition 7.2]{ncgstuff} for details.}.
  \end{itemize}  
  \tqed
\end{rmk}

\subsection{The Noetherian case}\

In situation of \Cref{thm:cat-main} if we further assume that the heart of $\CC$ is Noetherian, then we can draw stronger conclusions about $K$-theory and the induced $t$-structure on $\DD$.

\begin{lem} \label{lem:noetherian-transfer}
  In situation of \Cref{thm:cat-main}, if $\CC^\heartsuit$ is Noetherian, then the heart of the induced $t$-structure on $\DD$ is Noetherian as well.
\end{lem}

\begin{proof}
  We would like to show that every $d \in \DD^\heart$ is Noetherian.
  By \Cref{prop:small-t}, $d$ has a finite filtration with associated graded in $F(\CC^{\heart})$, so since Noetherian objects are closed under extensions, it suffices to show that $f(c)$ is Noetherian for each $c \in \CC^{\heart}$.
  As argued in the proof of \Cref{thm:cat-main},
  $f$ is fully faithful with image closed under passing to subobjects. This implies that the lattice of subobjects of $f(c)$ agrees with that of $c$, so $f(c)$ is Noetherian since $c$ is.
\end{proof}

As a consequence of the vanishing theorems of \cite{schlichting2006negative} and \cite{antieau2018ktheoretic} the negative $K$-groups of categories with a bounded $t$-structure with Noetherian heart vanish. In our setting this lets us upgrade the equivalence of $K$-theory spectra to one of non-connective $K$-theory spectra
\[ K^{\mathrm{nc}}(\CC) \xrightarrow{\simeq} K^{\mathrm{nc}}(\DD). \]

Next we examine the interpretation of condition (B) as ``unipotence'' more closely.
At the moment we know that
\begin{enumerate}
\item $f: \CC^\heart \to \DD^\heart$ is fully faithful.
\item Each $d \in \DD^\heart$ has a finite filtration with associated graded in the image of $f$.
\item Every subobject of $f(c)$ comes from a subobject of $c$.
\end{enumerate}

Note that the finite filtrations are not guaranteed to be functorial\footnote{Quillen's devissage, doesn't require a functorial filtration, but merely an object-wise one.}, but when $\DD$ is Noetherian, functoriality comes for free in the form of the \textit{socle filtration}:

\begin{cnstr}
  At the level of $\Ind$-categories the functor $f$ has right adjoint given by $g \coloneqq \tau_{\geq 0}G(-)$.
  Since $f$ is a fully faithful left adjoint it exhibits $\Ind(\CC)^\heart$ as a coreflective subcategory of $\Ind(\DD)^\heart$. We let $\mathrm{soc}_0(M)$ denote $fg(M)$, which sits as a subobject of $M$ via the counit map.
   
  We define $\mathrm{soc}_n(-)$ inductively via the pullback
  \begin{center}
    \begin{tikzcd}
      & \pullback\mathrm{soc}_n \ar[d, hook] \ar[r, two heads] & \mathrm{soc}_0(\mathrm{Id}/\mathrm{soc}_{n-1}) \ar[d, hook] \\
      \mathrm{soc}_{n-1} \ar[r, hook] & \mathrm{Id} \ar[r, two heads] & \mathrm{Id}/\mathrm{soc}_{n-1}.
    \end{tikzcd}
  \end{center}
  The key property of the socle filtration is that $\mathrm{soc}_n/\mathrm{soc}_{n-1}$ is in the image of $f$ for every $n$.
  Since the maps
  $ \mathrm{Id}/\mathrm{soc}_{n-1} \to \mathrm{Id}/\mathrm{soc}_{n} $
  become zero after applying $g$ and since $g = \tau_{\geq 0}G$ and $G$ detects coconnectivity we learn that
  \[\colim_n \mathrm{soc}_n \to \mathrm{Id} \]
  is an equivalence, i.e. the socle filtration is exhaustive.

  Now, using our Noetherian-ness hypothesis we know that arbitrary subobjects of compact objects are compact in $\Ind(\DD)^\heart$, therefore the socle filtration restricts to a socle filtration on $\DD^\heart$. 
  \tqed
\end{cnstr}

\begin{rmk} 
  The construction of a bounded $t$-structure together with socle filtrations can be interpreted as a generalization of \cite[Theorem 8.1]{keller2011weight} where a similar result is proven for dg-algebras with strong finiteness assumptions.
  \tqed
\end{rmk}

It is important to note that outside of the Noetherian setting the socle filtration need not restrict to compact objects.

\begin{exm}
  Let $\CC$ denote the category of pairs $(V_0,V_1)$ where
  $V_0$ is a $k[x_1,\dots]$-module and
  $V_1$ is a $k$-module.
  Since the infinite dimensional affine space is regular coherent, $\CC$ has a bounded $t$-structure.
  Let $\DD$ denote the category of triples
  $(V_0,\ V_1,\ V_0 \otimes_{k[x_1,\dots]} k \to V_1)$.

  The natural functor $\CC \to \DD$ which uses the zero map is fully faithful on the heart, therefore the hypotheses of \Cref{thm:cat-main} are satisfied. On the other hand, the socle of
  \[ (k[x_1,\dots],\ k,\ k[x_1,\dots] \twoheadrightarrow k) \]
  is $(I, k)$ and the augmentation ideal of $k[x_1,\dots]$ is not compact.
  \tqed
\end{exm}

\subsection{Negative \texorpdfstring{$K$}{K}-theory}\label{stablecoh}\

It would be desirable to extend \Cref{thm:cat-main} to negative $K$-theory, however both the theorem of the heart and Quillen's devissage only apply to connective $K$-theory in their current form. For that reason we ask the following question (which we hope has a positive answer):

\begin{qst} \label{qst:neg-heart}
  Do the theorem of the heart and devissage hold for negative $K$-theory?
\end{qst}

As discussed above, if $\CC^{\heartsuit}$ is Noetherian, then $\DD^{\heartsuit}$ is Noetherian as well and the negative $K$-theory of both categories vanishes by \cite{antieau2018ktheoretic}.
This might suggest that one should approach this question by proving a vanishing statement for negative $K$-theory.
However, the example from \cite{neeman} shows that in general regularity does not imply the vanishing of negative $K$-groups.

In order to probe question of this type more closely we examine the relation between \emph{stable coherence} and the vanishing of negative $K$-theory in \cite[Section 3.2]{ncgstuff}\footnote{see also \cite[Section 3.5]{antieau2018ktheoretic} for a similar discussion}. We reproduce the statements proved therein for the convenience of the reader interested in thinking about \Cref{qst:neg-heart}.

\begin{dfn}
  Given a category $\CC$ with bounded $t$-structure we say that $\CC$ is \emph{$\A^n$-coherent} if the finitely presented $\Z[t_1,\dots,t_n]$-modules in $\Ind(\CC^{\heartsuit})$ form an abelian category\footnote{If $\CC$ is the category of perfect $R$-modules for $R$ a discrete ring, then this is equivalent to asking that $R[t_1,\dots,t_n]$ be left coherent.}. If $\CC$ is $\A^n$-coherent for all $n$, then we say it is \emph{stably coherent}. \todo{Is this equivalent to what we do elsewhere?}
  \tqed
\end{dfn}

\begin{exm}
  In \cite[Example 7.3.13]{glaz2006commutative}, it is shown that an infinite product of copies of the ring $\Q[\![x,y]\!]$ is regular coherent, but this ring is not $\A^1$-coherent (demonstrating that this is a non-trivial condition on a regular coherent ring).
  \tqed
\end{exm}

The following lemma uses the vanishing results of \cite{antieau2018ktheoretic}.

\begin{lem}[{\cite[Corollary 3.14, Lemma 3.17]{ncgstuff}}]
  If $\CC$ is regular and $\A^n$-coherent, then the first $n+1$ negative $K$-groups of $\CC$ vanish.
\end{lem}

\begin{prop}[{\cite[Proposition 3.18]{ncgstuff}}]
  If we are given a functor
  $ F : \CC \to \DD \in \Cat^{\perf}$ and
  a bounded $\A^n$-coherent $t$-structure on $\CC$
  such that the conditions of \Cref{thm:cat-main} are satisfied,
  then the induced $t$-structure on $\DD$ is $\A^n$-coherent as well.
\end{prop}


\section{$K$-theory of pushouts}
\label{sec:applications}

In this section we prove \Cref{thm:pushouts} (a corollary of which is \Cref{thm:discrete-pushouts} from the introduction) which says that $K$-theory preserves pushouts of (well-behaved) regular prestable categories.
This theorem arises from the examining the interaction of our main theorem with the Land--Tamme $\odot$-product.
In fact, the idea that a result like \Cref{thm:opener} should be true was suggested to the authors by Markus Land and Georg Tamme with the intention of using such a result to compute the $K$-theory of pushouts.

\subsection{A review on the \texorpdfstring{$\odot$}{odot}-product}\ 

The Land--Tamme $\odot$-product is a relatively new operation on $\E_1$-algebras (and categories more generally) first introduced in \cite{Land_2019}, with generalizations appearing in \cite{bachmann2020categorical} and the forthcoming \cite{LTformula2021}. Here we roughly follow \cite{LTformula2021} in our formulation of this operation. 

\begin{cnstr}\label{cnstr:odot}
  Given a pair of categories $\BB$ and $\CC$ in $\Cat^{\perf}$ and an arrow $f \in \Fun^L(\Ind(\CC), \Ind(\BB)$ we can form the oplax limit of $f$, which we denote $\BB \vec{\times}_f \CC \in \Cat^{\perf}$.
  This is the category of triples $(b,c,r)$, where $b \in \BB$, $c \in \CC$, and $r: b \to f(c)$ is a map.
  For our purposes, the key observation about $\BB \vec{\times}_f \CC$ is that the forgetful map
  $\BB \vec{\times}_f \CC \to \BB \times \CC$
  induces an equivalence at the level of nc motives
  \[ \uloc(\BB \vec{\times}_f \CC) \cong \uloc(\BB \times \CC) \cong \uloc(\BB) \oplus \uloc(\CC). \]
  Adding another layer,
  associated to each square
  \begin{center}
    \begin{tikzcd}
      \AA \ar[r] \ar[d] & \CC \ar[d, "f"]  \\
      \BB \ar[r] \ar[ur, Rightarrow] & \Ind(\BB).
    \end{tikzcd}
  \end{center}
  we have an induced map $\AA \to \BB \vec{\times}_f \CC$ in $\Cat^{\perf}$
  and we define the Land--Tamme $\odot$-product $\BB \odot^{f}_{\AA} \CC$ to be the cofiber of this map.
  This cofiber sequence in $\Cat^{\perf}$ provides a pushout in nc motives
  \begin{center}
    \begin{tikzcd}
      \uloc(\im \AA)\ar[r]\ar[d] & \uloc(\CC) \ar[d]\\
      \ar[r] \uloc(\BB) & \uloc(\BB \odot_{\AA}^{f} \CC)\pushout
    \end{tikzcd}
  \end{center}
  where $\im\AA$ denotes the image of $\AA$ inside $\BB \vec{\times}_f \CC$.
  \tqed
\end{cnstr}

\begin{rmk}
  If we specialize to the case where $\AA$, $\BB$ and $\CC$ are module categories of $\E_1$-algebras $A$, $B$ and $C$, then we can move down a categorical level:
  \begin{enumerate}
  \item The category $\Fun^L(\Mod(C), \Mod(B))$ can be identified with $\Mod(B \otimes C^{\op})$,
    meaning the bonding map is just a choice of ($B$, $C$)-bimodule $M$.

  \item $\im\AA$ is generated by the image of $A$, meaning $\im(\AA) \cong \Mod(\im A)^{\omega}$ where $\im A$ is the endomorphism algebra of the image of $A$.
  \item If the functors $\AA \to \BB$ and $\AA \to \CC$ came from $\E_1$-algebra maps $A \to B$ and $A \to C$, then $\BB \odot_{\AA}^{f} \CC$ is generated by the image of $B$ (which is equivalent to the image of $C$). We let $B \odot_{A}^{M} C$ denote the ring of endomorphisms of this object where $M$ is the bimodule used as the bonding map.
  \end{enumerate}
  \tqed  
\end{rmk}

Speaking practically, the fundamental difficulty in working with the $\odot$-product lies in identifying the categories $\im (\AA)$ and $\BB \odot_{\AA}^{f} \CC$.
A fundamental insight of Land and Tamme is that in many cases of interest these categories are surprisingly computationally accessible.

\begin{exm}\label{exm:LT-old}
  In \cite{Land_2019}, where the $\odot$-product was introduced,
  the following example of \Cref{cnstr:odot} is analyzed.
  Suppose we are given a pullback square of $\E_1$-algebras
  \begin{center}
    \begin{tikzcd}
      A \ar[r] \ar[d] \pullback & C \ar[d] \\
      B \ar[r] & D.
    \end{tikzcd}
  \end{center}
  Using $D$ as our $(B,C)$-bimodule,
  $\Mod(A)$ for $\AA$ and 
  the map $ B \to D $ of ($B$, $A$)-bimodules for the natural transformation, we can construct a $\odot$-product $  B \odot^{D}_{A} C $.
  In this situation they prove that $\im(A) \cong A$ and
  the spectrum underlying $B \odot^{D}_{A} C$ is equivalent to $B \otimes_A C$ as a $(B,C)$-bimodule. Furthermore, they show in \cite[Proposition 1.13]{Land_2019} that the underlying $C$-bimodule of $B \odot^{D}_{A} C$ is the cofiber of the map $I\otimes_AC\to C$, where $I$ is the fiber of $C \to D$.
  \tqed
\end{exm}

In the forthcoming \cite{LTformula2021} another, somewhat dual, situation is analyzed.

\begin{thm}[\cite{LTformula2021}]\label{thm:catpushoutLT}
  Given a span $\BB \xleftarrow{b} \AA \xrightarrow{c} \CC$ in $\Cat^{\perf}$
  there is a square
  \begin{center}
    \begin{tikzcd}[sep=huge]
      \AA \ar[r, "c"] \ar[d, "b"] & \CC \ar[d, "b^*c_*"]  \\
      \BB \ar[r] \ar[ur, Rightarrow, "b^* \eta_c"'] & \Ind(\BB)
    \end{tikzcd}
  \end{center}
  and an equivalence of the associated $\odot$-product with the pushout of the span,
  \[ \BB \odot^{b^*c_*}_{\AA} \CC \simeq \BB \coprod_{\AA} \CC. \]
\end{thm}

\begin{cor}[\cite{LTformula2021}]\label{thm:LTformula}
  Given a span $B \leftarrow A \rightarrow C$ of $\E_1$-algebras we have an equivalence 
  \[ B \odot^{B \otimes_A C}_{A} C \simeq B \coprod_{A} C \]
  of the $\odot$-product with the pushout of the span in $\E_1$-algebras.
\end{cor}

\begin{rmk} \label{rmk:identify-ima}
  In \Cref{thm:LTformula} if we have a span of commutative algebras instead, then
  the base-change equivalence $B \otimes_A C \otimes_C - \cong B \otimes_A -$ allows us to recognize that we are actually in the situation of \Cref{exm:LT-old} for the cospan $B \to B \otimes_A C \leftarrow C$.
  The benefit of making this identification is that we can identify $\im(A)$ as the pullback of this cospan.
  \tqed
\end{rmk}

\begin{lem}\label{lem:LTbasechange}
  The $\odot$-product is compatible with base-change, i.e.
  \[ (\BB \odot^f_{\AA} \CC) \otimes \DD \cong (\BB \otimes \DD)\odot^{f \otimes \DD}_{\AA \otimes \DD}(\CC \otimes \DD). \]
\end{lem}

\begin{proof}
  This follows from the fact that $- \otimes \BB$ preserves
  fully faithful maps and localization sequences (see for example \cite[Lemma 3.3, Corollary 3.5]{antieau2018ktheoretic})
  and commutes with pullbacks, lax pullbacks and oplax limits of arrows.
\end{proof}

The $\odot$-product allows us to produce examples of equivalences of nc motives which do not arise from equivalences of categories.
We end our recollection by working through a pair of examples which illustrate this flexibility phenomenon.

\begin{dfn}
  Let $\Ss[x_n]$ denote the polynomial algebra on a generator in degree $n$.
  We let $\nil_n$ denote the cofiber
  \[ \nil_n \coloneqq \cof\left( \uloc(\Ss) \to \uloc(\Ss[x_{n}])\right) \]
  and use $\nil$ to mean $\nil_0$.
  \tqed
\end{dfn}

If we think about $\nil$ as a homology theory on nc motives it is the nil-$K$-theory of Bass which measures the failure of $\A^1$-invariance.

\begin{exm} \label{exm:a1-issue}
  Consider the span of commutative algebras $\Ss \leftarrow \Ss[x_n] \to \Ss$.
  Applying \Cref{thm:LTformula} we obtain a pullback of nc motives
  \begin{center}
    \begin{tikzcd}
      \pullback \uloc(\im \Ss[x_{n}]) \ar[d] \ar[r] & \uloc(\Ss) \ar[d] \\      
      \uloc(\Ss) \ar[r] & \uloc(\Ss[x_{n+1}]).
    \end{tikzcd}
  \end{center}
  Using \Cref{rmk:identify-ima} we can identify $\im (\Ss[x_n])$ as $\Ss[\epsilon_n]$ (the exterior algebra on a class in degree $n$) since $\Ss \otimes_{\Ss[x_n]} \Ss \simeq \Ss[\epsilon_{n+1}]$ and the pullback moves the exterior generator down a degree.
  Since the diagram above is diagonally symmetric we have a splitting 
  \[ \uloc( \Ss[\epsilon_n] ) \simeq \o \oplus \Sigma^{-1} \nil_{n+1}. \]
  \tqed
\end{exm}

We learned of next example, which allows us to turn a copy of the coordinate axes in the plane into a polynomial algebra, from Markus Land and Georg Tamme.

\begin{exm}\label{exm:coordinateaxesfreealg}
  Consider the algebra $R \coloneqq \Ss[x_a,x_b]/(x_ax_b)$
  which is built from the pullback square on the left below
  (where $x_a$ is in degree $a$ and $x_b$ is in degree $b$).
  \begin{center}
    \begin{tikzcd}
      \pullback \Ss[x_a,x_b]/(x_ax_b) \ar[r] \ar[d] & \Ss[x_{a}] \ar[d] &[-15pt] & \pullback \uloc(\Ss[x_a,x_b]/(x_ax_b)) \ar[r] \ar[d] & \uloc(\Ss[x_a]) \ar[d] \\
      \Ss[x_b] \ar[r] & \Ss & & \uloc(\Ss[x_b]) \ar[r] & \uloc\left(\Ss[x_a] \coprod_{\Ss[x_a,x_b]} \Ss[x_b] \right)
    \end{tikzcd}
  \end{center}
  Applying \Cref{thm:LTformula} and \Cref{rmk:identify-ima} we then obtain the pullback of nc motives on the right.
  In fact, we can simplify this by exhibiting an equivalence of $\E_1$-algebras
  \[ \Ss[x_a] \coprod_{\Ss[x_a,x_b]} \Ss[x_b] \simeq \Ss[x_{a+b+2}]. \]
  To show this, suppose first that $a,b > 0$. Then using the fact that
  \[ \Spaces_* \xrightarrow{k \otimes \Sigma_+^\infty \Omega -} \Alg_k \]
  is a left adjoint and so preserves pushouts, we reduce to the pushout square
  \begin{center}
    \begin{tikzcd}
      S^{a+1} \times S^{b+1} \ar[r] \ar[d] & S^{a+1} \ar[d] \\
      S^{b+1} \ar[r] & S^{a+b+3}.\pushout
    \end{tikzcd}
  \end{center}
  
  In order to extend this equivalence to the case where $a,b$ are not strictly positive we use a trick.
  First, we lift the pushout above to a pushout in graded rings where $x_a$ and $x_b$ are in grading $1$.
  Since the forgetful functor preserves colimits it will suffice to compute the pushout in the graded setting.
  Next we use the $\E_2$-monoidal shearing functor which suspends by $2n$ in grading $n$ constructed in \cite{rotation} to reduce to the case where $a,b$ are positive.

  In the graded setting the generator $x_{a+b+2}$ is in grading $2$ and as a consequence of the fact that $\Ss[x_a]$ is the free graded algebra on the class $x_a$ in grading $1$ we obtain a factorization
  $ \Ss[x_a] \to \Ss \to \Ss[x_{a+b+2}] $.
  With control over the maps in the square above we now obtain an equivalence of nc motives
  \[ \uloc(\Ss[x_a,x_b]/(x_ax_b)) \simeq \o \oplus \nil_a \oplus \nil_b \oplus \Sigma^{-1} \nil_{a+b+2}. \]
  \tqed
\end{exm}

\subsection{\texorpdfstring{$K$}{K}-theory of pushouts}\ 

There is a sharp contrast between the ideas behind the Land--Tamme $\odot$-product and our main theorem.
The $\odot$-product arises from $2$-categorical maneuvers and essentially operates at the level categories and nc motives.
Meanwhile our main theorem is specific to $K$-theory, exploiting additive but non-exact operations (such as truncation) in an essential way.
The complementary nature of these approaches allows us to combine them to surprising effect.

\begin{thm}\label{thm:pushouts}
  Suppose we are given a span $\BB \xleftarrow{b} \AA \xrightarrow{c} \CC$ in $\Cat^{\perf}$ where
  $\AA$ is equipped with a bounded $t$-structure.
  If we assume that
  \begin{enumerate}
  \item[(D)] the induced functor $\AA^{\heart} \to \BB \vec{\times}_{b^*c_*} \CC$ is fully faithful,
  \end{enumerate}  
  then connective $K$-theory preserves the pushout of the span,
  i.e the diagram below is a pushout square.
  \begin{center}
    \begin{tikzcd}
      K(\AA) \ar[r]\ar[d] & K(\CC)\ar[d]\\
      K(\BB)\ar[r] & \pushout K\left( \BB \coprod_{\AA} \CC \right)
    \end{tikzcd}
  \end{center}
\end{thm}

\begin{proof}
  From \Cref{thm:catpushoutLT}, we have a pushout square
  \begin{center}
    \begin{tikzcd}
      K^{\mathrm{nc}}(\im \AA) \ar[r]\ar[d] & K^{\mathrm{nc}}(\CC)\ar[d]\\
      K^{\mathrm{nc}}(\BB)\ar[r] & K^{\mathrm{nc}}\left( \BB \coprod_{\AA} \CC \right) \pushout.
    \end{tikzcd}    
  \end{center}
  Condition (D) implies that the functor $\AA \to \im(\AA)$ satisfies the hypotheses of \Cref{thm:cat-main}, so we have an equivalence $K(\AA) \simeq K(\im\AA)$.
  Moreover, since $\AA$ and $\im \AA$ each have a bounded $t$-structure \cite[Theorem 1.1]{antieau2018ktheoretic} implies that $K_{-1}(\AA) = 0 = K_{-1}(\im \AA)$.
  This implies that the square above remains a pushout when we take connected covers and replace $K(\im \AA)$ by $K(\AA)$.
\end{proof}

In order to make this theorem easier to apply we give a simpler condition which implies (D) and is more natural to check in practice.

\begin{lem}\label{lem:pushout-alt}
  In the situation of \Cref{thm:pushouts} condition (D) is implied by
  \begin{itemize}
  \item[(D$'$)] The functors $\AA^\heartsuit \to \BB$ and $\AA^\heartsuit \to \CC$ are faithful.
  \end{itemize}    
\end{lem}

\begin{proof}


  
  Let $F$ denote the functor
  $ \AA \to \BB \vec{\times}_{b^*c_*} \CC $.
  Using \Cref{lem:condbprime} it suffices to show that
  $\cof(a \to F_*F^*(a))$ is $\leq -1$ for each $a \in \AA^{\heartsuit}$.
  In order to proceed we'll need to give a formula for $F_*F^*$.
  From the pullback square
  \begin{center}    
    \begin{tikzcd}\pullback
      \Map_{\BB\vec{\times}_{b^*c_*}\CC}(F^*x,F^*y) \ar[rr] \ar[d] & & \Map_{\BB}(b^*x, b^*x) \ar[d] \\
      \Map_{\CC}(c^*x, c^*y) \ar[r, "b^*c_*"] & \Map_{\BB}(b^*c_*c^*x, b^*c_*c^*y) \ar[r, "b^*\eta_c \circ-"] & \Map_{\BB}(b^*x, b^*c_*c^*y)      
    \end{tikzcd}
  \end{center}  
  natural in both $x$ and $y$ we learn that $F_*F^*$ sits in a pullback square
  \begin{center}
    \begin{tikzcd}
      \pullback F_*F^* \ar[r] \ar[d] & b_*b^* \ar[d, "b_*b^* \circ \eta_c"] \\
      c_*c^* \ar[r, "\eta_b \circ c_*c^*"] & b_*b^*c_*c^*.
    \end{tikzcd}
  \end{center}
  We can then read off that
  \[ \cof(\Id \to F_*F^*) \simeq \Sigma^{-1} \cof(\Id \to b_*b^*) \circ \cof(\Id \to c_*c^*) \]
  where $\circ$ is the composition monoidal structure on $\Fun^{L}(\Mod(A),\Mod(A))$.
  Using \cite[Remark 4.9]{ncgstuff} (which is a variant of \Cref{lem:condbprime}) and compatibility with colimits we can reformulate the faithfulness hypothesis as saying that $\cof(\Id \to c_*c^*)$ and $\cof(\Id \to b_*b^*)$ preserve coconnectivity. Composing and desuspending we obtain the desired coconnectivity bound on $\cof(\Id \to F_*F^*)$.
\end{proof}

For discrete rings condition (D$'$) has a simple interpretation:
A map $A \to B$ is fully faithful on the heart exactly when $B$ is right faithfully flat as an $A$-module (see \cite[Lemma 4.7]{ncgstuff}). Consequently, we obtain the following corollary, which appeared in the introduction as \Cref{thm:discrete-pushouts}.

\begin{cor} \label{cor:discrete-pushouts}
  Suppose $B \xleftarrow{f} A \xrightarrow{g} C$ is a span of discrete rings
  where $A$ is left regular coherent and both $f$ and $g$ are right faithfully flat.
  Then connective $K$-theory preserves the pushout of this span.
\end{cor}

\begin{rmk}
  Note that (D) does not imply (D$'$).
  For example if $X$ and $Y$ are (well-behaved) smooth varieties which form a Zariski covering of $Z$, then (D) is satisfied for the span
  \[ \QCoh(X) \leftarrow \QCoh(Z) \rightarrow \QCoh(Y) \] 
  while (D$'$) need not be satisfied.
  \tqed
\end{rmk}

\begin{rmk}
  \Cref{cor:discrete-pushouts} (and in turn \Cref{lem:pushout-alt} and \Cref{thm:pushouts}) can be viewed as a generalization of \cite[Theorems 1 and 4]{Waldhausen-free} where the stronger condition that $f: A \to B$ and $g: A \to C$ are pure inclusions\footnote{This asks that $B$ have a splitting $B \cong f(A) \oplus I$ as an $A$-bimodule where $I$ is a projective right $A$-module.} was imposed.
  \tqed
\end{rmk}


\section{Applications and Examples}
\label{sec:examples}

In this section we work through a collection of applications and examples which use \Cref{thm:opener}.
Of particular note are
\begin{itemize}
\item[(]\hspace{-0.15cm}Prop.\ref{thm:a1invariance}) which proves $\A^1$-invariance for regular categories.
\item[(]\hspace{-0.15cm}Prop.\ref{thm:aninvariance}) which proves $\A^n$-invariance for regular categories in high degrees.
\item[(]\hspace{-0.15cm}Prop.\ref{thm:cochainalgs}) which analyzes the $K$-theory of unipotent local systems.
\item[(]\hspace{-0.15cm}Exm.\ref{exm:smalltorexm}, \ref{exm:An-doubled} and \ref{exm:hidden-sing}) which show that the conditions of \Cref{thm:opener} are sharp.
\end{itemize}

\subsection{Invariance theorems}\ 

We give a short proof of $\A^1$-invariance of $K$-theory for categories with a bounded $t$-structure.
This result was first proven for regular Noetherian rings by Quillen in his foundational paper \cite{quillenhigherktheory}.
Building on this we then prove that $K_j(-)$ is $\A^n$-invariant once $j \geq n-1$ (again for categories with a bounded $t$-structure).
Using \Cref{thm:pushouts} we then extend $\A^1$-invariance to the case of adjoining free variables generalizing the main results of \cite{gerstenfree}. 

\begin{prop}[$\A^1$-invariance for regular categories]\label{thm:a1invariance}\
  
  If $\CC \in \Cat^{\perf}$ admits a bounded $t$-structure, 
  then $K(\CC) \simeq K(\CC[x_0])$.
\end{prop}

\begin{proof}
  In order to prove this we must show that $K^{\nc}(\nil \otimes \CC)$ vanishes in non-negative degrees.
  Applying \Cref{thm:cat-main} to the map
  $ \CC \to \CC[\epsilon_{-1}] $
  and moving to the other side of the equivalence from \Cref{exm:a1-issue} we learn that $K^{\mathrm{nc}}( \nil \otimes \CC )$ vanishes in non-negative degrees as desired\footnote{In degree zero this uses that $K_{-1}$ of $\CC$ and $\CC[\epsilon_{-1}]$ both vanish.}.
\end{proof}

Just as $\nil$ controls $\A^1$-invariance, $\A^n$-invariance is controlled by tensor-powers of $\nil$.
Using the same ideas we can show that $K$-theory is $\A^n$-invariant in sufficiently large degrees as well.

\begin{prop}[$\A^n$-invariance for regular categories]\label{thm:aninvariance}
  Suppose $C \in \Cat^{\perf}$ admits a bounded $t$-structure.
  Then $\tau_{\geq n-1}K(C) \simeq \tau_{\geq n-1}K(C[x_1,\dots,x_n])$, where $|x_i| = 0$.
\end{prop}

\begin{proof}
  From the equivalence
  \[ \uloc(\Ss[x_1,\dots,x_n]) \simeq \uloc(\Ss[x]^{\otimes n}) \simeq (\uloc(\Ss[x]))^{\otimes n} \simeq ( \o \oplus \nil )^{\otimes n} \]
  we can read off that the obstructions to $\A^n$-invariance in degree $j$ are
  $K_j(\nil^{\otimes k} \otimes \CC)$
  for $1 \leq k \leq n$.
  Using \Cref{exm:a1-issue} we can find $\Sigma^{-k}\nil^{\otimes k} \otimes \CC$ as a summand in $\uloc(\CC[\epsilon_1,\dots,\epsilon_k])$ (where each exterior generator is in degree $-1$).
  Applying \Cref{thm:cat-main} to the map $\CC \to \CC[\epsilon_1,\dots,\epsilon_k]$ we learn that $K^{\nc}( \Sigma^{-k}\nil^{\otimes k} \otimes \CC )$ vanishes in degrees $\geq -1$, which lets us conclude. \qedhere
\end{proof}

\begin{cor}\label{cor:ringaninvariance}
  Let $R$ be a left regular coherent ring. Then $K_i(R) = K_i(R[x_1,\dots,x_n])$ for $i\geq n-1$.
\end{cor}

As mentioned in the introduction, the above corollary, which is more subtle for $n>1$, was already known when $R$ is a discrete ring, where it follows from the Farrell--Jones conjecture for the groups $\Z^n$.

\begin{prop}[Free generator invariance]\label{thm:free-invariance}
  Let $\CC \in \Cat^{\perf}$ have a bounded $t$-structure.
  Then $K(\CC) \simeq K(\CC\{x_1,\dots,x_n\})$, where $|x_i| = 0$.
\end{prop}

\begin{proof}
  We proceed by induction on $n$ with base-case given by \Cref{thm:a1invariance}.
  If we consider the pushout of categories
  \begin{center}
    \begin{tikzcd}
      \CC \ar[r, "i"] \ar[d, "j"] & \CC\{x_1,\dots,x_{n-1}\} \ar[d] \\
      \CC[x] \ar[r] & \CC\{x_1,\dots,x_n\} \pushout.
    \end{tikzcd}
  \end{center}
  then condition (D$'$) holds since the arrows labeled $i$ and $j$ each have a section (sending all the $x$'s to zero). As a consequence we can apply \Cref{thm:pushouts} and conclude.

\end{proof}

\begin{cor}\label{cor:ringa1invariance}
  Let $R$ be a left regular coherent ring,
  then $K(R) = K(R\{x_1,\dots,x_n\})$.
\end{cor}

In fact, \Cref{cor:ringa1invariance} is a special case of the next example,
which allows for a more general module in place of the indeterminants $x_1,\dots,x_n$.

\begin{exm}\label{exm:tensoralgebra} 
  Suppose that $R$ is a discrete, left regular coherent ring and $M$ is a discrete, right flat $R$-bimodule. We would like to apply \Cref{thm:pushouts} to the span of rings
  \[ R \leftarrow R\{\Sigma^{-1}M\} \to R \]
  whose pushout is $R\{M\}$ and conclude that $K(R) \simeq K(R\{M\})$.
  
  In order to check that the functor $\Mod(R\{\Sigma^{-1}M\}) \to \Mod(R)$ is faithful on the heart we argue as follows: 
  Apply \Cref{thm:opener} to the map $R \to \R\{\Sigma^{-1}M\}$,
  as a consequence of proof of this theorem any object in the heart is an extension of induced objects.
  Examining long exact sequences it suffices to prove faithfulness on the heart for these induced objects.
  This follows from noting that the composite, $R \to R\{\Sigma^{-1}M\} \to R$ is just the identity.
  \tqed
  

\end{exm}

There are many more invariance-type results that can be proven using a combination of \Cref{thm:cat-main} and the Land-Tamme $\odot$-product and we end this subsection with a more generic example.

\begin{exm} \label{exm:other-invariance}
  Suppose we are given a map of discrete rings $R \to S$ with $R$ left regular coherent and an $S$-bimodule $M$ which is right flat over $R$.
  We can form the pullback of $\E_1$-algebras
  \begin{center}
    \begin{tikzcd}
      R\oplus \Sigma^{-1}M \ar[r]\ar[d] & S \ar[d]\\
      R\ar[r] & S\oplus M
    \end{tikzcd}
  \end{center}
  where $S\oplus M$ is a square-zero extension of $S$ by $M$ and $R\oplus\Sigma^{-1}M$ is the square-zero extension of $R$ by $\Sigma^{-1}M$.
  From \Cref{thm:opener} we know that $K(R\oplus\Sigma^{-1}M) \simeq K(R)$ and therefore
  \[ K(S) \simeq K \left( R\odot^{S\oplus M}_{R\oplus \Sigma^{-1}M}S \right). \]

  The underlying $(R,S)$-bimdoule of the $\odot$-product is given by $R \otimes_{R \oplus \Sigma^{-1}M} S$ which is equivalent to $R\{M\}\otimes_RS$. The free algebra $R\{M\}$ is discrete and right flat as an $R$-module since $M$ is, so $R\{M\}\otimes_RS$ is discrete as well.
  By the previous example, which is the case $R=S$, the $\odot$-product receives ring maps from both $R\{M\}$ and $S$.
  It remains then to determine the left multiplication of an element of $S$ by one of $M$. This can be read off using \Cref{exm:LT-old}, which gives a cofiber sequence of $S$-bimodules
  \[S \to R\odot^{S\oplus M}_{R\oplus \Sigma^{-1}M}S \to M\otimes_{R\oplus \Sigma^{-1}M}S,\]
  showing that the left multiplication of $S$ on $M$ is the one coming from the left $S$-module structure.    
  \tqed
\end{exm}

Note that in \Cref{exm:other-invariance} the ring $S$ is not required to be regular!
For example, we can let $R=k$ be a field, take $S = k[\epsilon]/\epsilon^2$ and let $M$ be $k$ thought of as an $S$-bimodule via the augmentation. In this case we obtain an equivalence
\[ K(k[\epsilon]/\epsilon^2) \simeq K(k\{\epsilon,y\}/(\epsilon^2,\epsilon y)). \]

\subsection{\texorpdfstring{$K$}{K}-theory of unipotent representations}
\label{subsec:locsys}\ 

Next we analyze the $K$-theory of categories of local systems with values in a regular category. The following generalizes the discussion in \cite[Section 4.3]{antieau2018ktheoretic}, in which they analyze the $K$-theory of cochain algebras of finite, connected spaces with coefficients in commutative Noetherian rings using the Koszul dual description of the module categories in \cite[Proposition 7.8]{mathew2016galois} as ind-unipotent representations of the loopspace.

\begin{dfn}
  Given $\CC \in \Cat^{\perf}$ and an $X \in \Spaces$, let $\Rep(X; \CC)$
  \todo{I'm slightly bothered by the fact that it is representations of the loopspace of X if X is connected, not X itself.} denote the category of local systems on $X$ with values in $\CC$ (this is just $\Fun(X, \CC)$)\footnote{There is a subtlety here, which is that in general $\Rep(X; \CC)$ and $\Fun(X, \Ind(\CC))^{\omega}$ differ. It is this which motivated us to use $\Rep$ as notation when $\Fun$ would appear to suffice.}. Pullback along the map $X \to *$ provides a functor
  \[(-)^{\mathrm{triv}} : \CC \to \Rep(X; \CC) \]
  which associates to $c \in \CC$ the constant local system at $c$.
  Let $\Rep(X; \CC)^{\mathrm{uni}}$ denote $\im((-)^{\mathrm{triv}})$.
  We refer to this as the category of unipotent local systems valued in $\CC$.
  \tqed
\end{dfn}

\begin{rmk}
  One can make similar definitions for $\AA$ a small abelian category. Namely, we let $\Rep(X;\AA)$ denote $\Fun(X,\AA)$, and let $\Rep(X,\AA)^{\mathrm{uni}}$ denote unipotent representations, i.e the category generated under extensions, kernels and cokernels by the image of $(-)^{\mathrm{triv}}$.
  \tqed
\end{rmk}

\begin{prop}\label{thm:cochainalgs}
	If $\CC\in \Cat_{\geq0}^{\perf}$ is bounded and regular, and $X$ is connected, then
	\begin{enumerate}
		\item Truncation on $\CC$ provides $\Rep(X; \CC)$ a bounded $t$-structure with heart $\Rep(B\pi_1X,\CC^{\heart})$.
		\item The $t$-structure on $\Rep(X; \CC)$ restricts to $\Rep(X; \CC)^{\mathrm{uni}}$, with heart $\Rep(B\pi_1X,\CC^{\heart})^{\mathrm{uni}}$.
		\item $(-)^{\mathrm{triv}}$ induces an equivalence $K(\CC)\simeq K(\Rep(- ; \CC)^{\mathrm{uni}}) $.
	\end{enumerate}  
\end{prop}

\begin{proof}
	For (1), in order to check that $\tau_{\geq 0}$ and $\tau_{<0}$ determine a $t$-structure on $\Rep(X; \CC)$ we just need to check that the space of maps from $\tau_{\geq 0}c$ to $\tau_{<0}d$ is contractible. To do this we use the formula
	\[ \Map_{\Rep(X;\CC)}(\tau_{\geq 0}c, \tau_{<0}d) \simeq \Map_{\CC}(\tau_{\geq 0}c, \tau_{<0}d)^{h\Omega X} \]
	where $\Omega X$ acts on the space of maps in $\CC$ by conjugation.
	Since $\Map_{\CC}(\tau_{\geq 0}c, \tau_{<0}d)$ is contractible, so is the limit under the action.
	Boundedness is inherited from $\CC$ since the underlying object functor $\Rep(X; \CC) \to \CC$ is $t$-exact and conservative. The heart is clearly $\Rep(X;\CC^{\heart})$, and since $\CC$ is a $1$-category, this is the same as $\Rep(\tau_{\leq1}X;\CC^{\heart})$. We conclude since $\tau_{\leq1} X \cong B\pi_1X$.
	
	For (2) and (3), we check that the functor $(-)^{\mathrm{triv}}$ is unipotent (see \Cref{dfn:ncg-defs}), so that we can apply \Cref{thm:cat-main} to conclude. Quasi-affineness follows from construction, and fully faithfulness on the heart follows from the fact that equivariant maps between objects in $\CC^{\heart}$ with trivial $\pi_1(X)$-action are just given by the underlying maps in $\CC^{\heart}$ since it is a $1$-category.
\end{proof}

\begin{rmk}
  If $R$ is an $\E_1$-algebra, then the $R$-module $R$ (with trivial action) is a generator of $\Rep(X; \Mod(R))^{\mathrm{uni}}$, therefore we may identify this category with $\Mod( C^*(X; R) )$, the category of modules over the cochain algebra of $X$ with values in $R$. \Cref{thm:cochainalgs}(3) provides an equivalence
  \[ K(R) \simeq K(C^*(X; R)). \]

  When $X$ is additionally compact and $R$ Noetherian and commutative, the above result combined with \Cref{lem:noetherian-transfer} and vanishing of negative $K$-theory coincides with \cite[Theorem 4.8]{antieau2018ktheoretic}. 
  However, as pointed out to us by Markus Land, their proof is not quite correct, since they claim that the heart of the $t$-structure on $C^*(X;R)$ agrees with that of $R$, which is not true if for example $X = S^1$. Despite this, the hearts are sufficiently similar that Quillen's devissage provides an equivalence on $K$-theory.
  \tqed
\end{rmk}

\subsection{Testing the limits of \Cref{thm:opener}}\ 

In the next sequence of examples we probe the limits of \Cref{thm:opener}.
Summarizing what we find: the conditions of \Cref{thm:opener} are sharp.
To see that regularity of $\pi_0R$ is necessary we look at an example where $\A^1$-invariance fails.

\begin{exm} \label{exm:ugh}
  We consider the exterior algebra $k[\epsilon_{0}, \epsilon_{-1}]$ over a field $k$.
  From \Cref{exm:a1-issue} we have an equivalence of non-connective $K$-theories
  \[ K^{\nc}(k[\epsilon_{0}, \epsilon_{-1}]) \simeq K^{\nc}(k[\epsilon_0]) \oplus \Sigma^{-1} K^{\nc}(\nil_0 \otimes k[\epsilon_0]). \]
  Since $\A^1$-invariance fails for $k[\epsilon_0]$ (see \cite{Hesselholt-Madsen}), both terms in the sum are non-trivial.
  On the other hand \Cref{thm:opener} predicts only the first term.
  \tqed  
\end{exm}

Now we turn to the tor condition of \Cref{thm:opener}. Essentially the simplest example of an algebra which violates it is the trivial square zero extension $S \coloneqq \F_p[x]\oplus \Sigma^{-1}\F_p$ where $x$ is in degree zero and acts by zero on $\F_p$.
In a conversation with Markus Land and Georg Tamme we determined that $K_1(S)$ differs from $K_1$ of $\F_p[x]$ by using the $\odot$-product to reduce to a connective ring and then using trace methods\footnote{A similar analysis also works for $S=\Z\oplus \Sigma^{-1}\F_p$}.

\begin{exm}\label{exm:smalltorexm}  
  The algebra $S$ fits into the pullback square on the left.
  \begin{center}
    \begin{tikzcd}
      \pullback S \ar[r] \ar[d] &\F_p[x]\ar[d] & & K^{\nc}(S)\pullback \ar[r] \ar[d] & K^{\nc}(\F_p[x]) \ar[d] \\
      \F_p[x] \ar[r] & \F_{p}[x]\oplus \F_p & & K^{\nc}(\F_p[x]) \ar[r] & K^{\nc}\left(\F_p[x]\{\F_p\} \right)
    \end{tikzcd}
  \end{center}
  Writing the $\F_p[x]$-bimodule $\F_p[x]\oplus \F_p$ as the tensor product\footnote{(Here $\{M\}$ denotes the free algebra on a bimodule.} $\F_p[x]\otimes_{\F_p[x]\{\Sigma^{-1}\F_p\}}\F_p[x]$  we can apply \Cref{thm:LTformula} to identify $\F_p[x] \odot_{\F_p[x]\{\Sigma^{-1}\F_p\}}^{\F_{p}[x]\oplus \F_p} \F_p[x] $ with $\F_p[x]\{\F_p\}$. From this we obtain the pullback square of $K$-theories on the right.
  
  Using $\A^1$-invariance we have an isomorphism of relative $K$-theories
  \[ \cof( K^{\nc}(\F_p) \to K^{\nc}(S) ) \simeq \Sigma^{-1} \cof( K^{\nc}(\F_p) \to K^{\nc}(\F_p\{\F_p\}) \]
  To conclude that $K_1(S)$ differs from $K_1(\F_p)$ we will argue that $K_2(\F_p\{\F_p\})$ is not even finitely generated.

  Let $R \coloneqq \F_p\{\F_p\}$. 	
  We can construct a DGA model for $R$ which is $\F_p[x]\{y,z\}$ with $|y| = 0, |z|=1$, $d(z)=xy$.
  From this we can compute that
  \begin{itemize}
  \item $\pi_0R \cong \F_p[x,y]/xy$,
  \item $x$ acts by zero on $\pi_1R$ and
  \item $\pi_1R$ is a free $\F_p$-vector space on the classes $y^{a_0}[z,y]y^{a_1}$ with $a_0,a_1\geq 0$.
  \end{itemize}
	
  As a consequence of Waldhausen's calculation of the first nonzero vanishing homotopy group of the fiber of $K(A) \to K(\pi_0A)$ for a connective simplicial ring $A$ (\cite[Proposition 1.2]{Waldhausen}), we learn that the fiber of $K(R) \to K(\pi_0R)$ is $1$-connected, and has second homotopy group given by
  \[ \mathrm{HH}_0(\F_p[x,y]/xy; \pi_1 R) \cong \F_p\{ y^{a}[z,y] \ |\  a\geq0 \}. \]
  Since $K_3(\F_p[x,y]/xy)$ is finitely generated (see \cite{Hesselholt_2007}) we learn that $K_2(R)$ is not finitely generated as promised.
  \tqed
\end{exm}

In the example above, although the $K$-theory differs from that of the connective cover, if we think in terms of \Cref{thm:cat-main} it is not immediately clear at what point things broke down. Possibilities include:
\begin{itemize}
\item The ring failed to be regular (in the sense of \Cref{dfn:ncg-defs}).
\item The base-change functor from the connective cover has failed to be $t$-exact.
\item The base-change failed to be fully faithful on the heart.
\end{itemize}
In view of this we now proceed to give several more geometric examples where we have better control over how things break down.

\begin{exm} \label{exm:An-minus-origin}
  Consider the quasi-affine variety $X \coloneqq \A_k^n \setminus \{0\}$ over a field $k$.
  Since this scheme is quasi-affine, its category of quasicoherent sheaves is equivalent to the category of modules over the ring of global sections, $R$. This is a commutative $k$-algebra whose homotopy groups are the coherent cohomology groups of $\A_k^n \setminus \{0\}$.
  
  In this case we have
  \[ \pi_sR \cong \begin{cases} k[x_1,\dots,x_n] & s=0 \\ (\prod_i x_i^{-1})k[x_1^{-1},\dots,x_n^{-1}] & s=1-n \\ 0 & \text{otherwise} \end{cases}. \]
  The divisible module which shows in degree $1-n$ has tor dimension $n$ and therefore violates condition (2) in \Cref{thm:opener}.
  Applying excision to the scissor congruence $* \to \A_k^n \leftarrow (\A_k^n - 0)$ allows us to conclude that
  \[ K(R) \simeq K(\A_k^n) \oplus \Sigma K(k) \]
  with the comparison map $\pi_0R \to R$ inducing the inclusion of the left summand.
  \tqed
\end{exm}

In \Cref{exm:An-minus-origin}, the tor condition fails and the $K$-theories differ, but $R$ is regular anyway.
What happens here is that the map $\Mod(\pi_0R)^{\heart} \to \Mod(R)^{\heart}$ isn't faithful because the module $k$ supported at the origin is sent to zero.
This example also exhibits another more subtle behavior.
In \cite[Proposition 1.1]{Waldhausen} (which is extended to general connective ring spectra by \cite[Lemma 2.4]{Land_2019}), Waldhausen shows that an $n$-connective map of connective algebras induces an $(n+1)$-connective map on $K$-theory. A similar phenomenon does not occur our setting.
In \Cref{exm:An-minus-origin} the first degree where $R$ differs from its $\pi_0R$ is $1-n$ while the $K$-theory first differs in degree $1$, which is independent of the parameter $n$.

Since \Cref{exm:An-minus-origin} isn't tight with respect to the tor condition we now provide another family of examples which, although more geometrically degenerate, do show that the tor condition is tight. 

\begin{exm} \label{exm:An-doubled}
  Consider $\A^n$ with a doubled origin over the same field $k$,
  that is to say we look at the pullback below
  \begin{center}
    \begin{tikzcd}
      \pullback R \ar[r] \ar[d] & \Gamma(\A^n) \ar[d] \\
      \Gamma(\A^n) \ar[r] & \Gamma(\A^n - 0).
    \end{tikzcd}
  \end{center}
  From our examination of $\Gamma(\A^n - 0)$ in \Cref{exm:An-minus-origin} we know that $\pi_{-n}R$ has tor dimension $n$.
  Since $\Gamma(\A^n) \to \Gamma(\A^n - 0)$ is a localization the induced square on $K$-theory is a pullback (see \cite{Tamme}). From this we can read off that
  \[ K(R) \cong K(\A_k^n) \oplus K(k) \]
  where again the comparison map $\pi_0R \to R$ induces the inclusion of the left summand.
  \tqed
\end{exm}

In the previous two examples the $K$-theory of $R$ and its connective cover differed, but $R$ was still regular (in the sense of \Cref{dfn:ncg-defs}).
Our next example will show that it is possible for $\pi_0R$ to be regular while $R$ is non-regular.
To do this we use an affine nodal cubic curve $C$ over a field $k$, which has non-vanishing $K_{-1}$  (see \cite[III.4.4]{weibel2013k}). The main result of \cite{antieau2018ktheoretic} then implies that the category of perfect coherent sheaves on $C$ is not regular.

\begin{exm} \label{exm:hidden-sing}
  Consider the nodal cubic curve $C \coloneqq \Spec ( k[x,y]/y^2-x^2(x-1) )$ and let $R$ denote the pullback below
  \begin{center}
    \begin{tikzcd}
      \pullback R \ar[r] \ar[d] & k[x^{\pm 1}] \ar[d] \\
      C \ar[r] & C[x^{-1}].
    \end{tikzcd}
  \end{center}
  Geometrically, the bottom horizontal arrow corresponds to removing the nodal point from $\Spec(C)$ and the right vertical arrow corresponds to the quotient of $C$ minus the node by the $C_2$ action that sends $y$ to $-y$.
  The homotopy groups of $R$ are 
  \[ \pi_sR = \begin{cases} k[x] & s=0 \\ (yx^{-1})k[x^{-1}] & s=-1 \\ 0 & \text{otherwise} \end{cases}. \] 
  
  As in the previous example, since the top horizontal arrow is a localization the induced square on $K$-theory is a pullback \cite{Tamme}.
  Since $C[x^{-1}]$ is regular and Noetherian its negative $K$-groups vanish.
  This implies that $K_{-1}(R) \cong K_{-1}(C) \neq 0$. 
  On the other hand $\pi_0R = k[x]$ has vanishing $K_{-1}$.
  Since $K_{-1}(R)$ is non-zero, the category of compact $R$-modules cannot be regular.
  \tqed
\end{exm}

This example demonstrates that when the tor condition is violated, $\Mod(R)^{\omega}$ can fail to be regular in addition to the $K$-theories of $R$ and $\pi_0R$ differing.
In essence what we have done in \Cref{exm:hidden-sing} is taken a singularity and hidden it in degree $-1$. The fact that this can be done implies that the tor condition in \Cref{thm:opener} and the condition that $\pi_0R$ be left regular coherent cannot be disentangled---an idea we explore further in the next example.

\begin{exm}
  Suppose that $R$ is a regular, discrete, Noetherian, commutative algebra.
  Using \Cref{thm:cat-main}, \Cref{exm:a1-issue}, \Cref{exm:coordinateaxesfreealg} and \Cref{thm:a1invariance} we obtain $K$-theory equivalences
  \begin{align*}
    K(R[x_0,x_{-1}]/(x_0x_{-1})) 
    &\simeq K(R) \oplus K(\nil_0 \otimes R) \oplus K(\nil_{-1} \otimes R) \oplus \Sigma^{-1}K(\nil_{1} \otimes R) \\
    &\simeq K(R[x_0]) \oplus K(R[\epsilon_0]) \simeq K(R[\epsilon_0]). 
  \end{align*}
  \tqed
\end{exm}

What distinguishes this example is that
$R[x_0,x_{-1}]/(x_0x_{-1})$ is coconnective, has Noetherian, regular $\pi_0$, but violates the tor condition, while
$R[\epsilon_0]$ is discrete (and therefore satisfies the tor condition), but is non-regular.
This suggests that at the level of nc motives
regularity of $\pi_0R$ and the tor condition are not individually particularly meaningful.
Instead we should think of the combination of these two conditions, i.e. unipotence,  as a meaningful single condition.

\subsection{An example we do not cover}
\label{sec:notcovered}\

We end the paper by giving a simple example of a ring $R$ such that
the connective cover map $\pi_0R \to R$ induces an equivalence on nc motives,
but $R$ does not satisfy the hypotheses of \Cref{thm:opener}.

\begin{exm}
  Let $R$ be the ring $\End_{\F_p[x,y]}((x,y))^{op}$.
  As a module over $\F_p[x,y]$, $(x,y)$ has three cells, two in degree $0$, and one in degree $1$ with attaching maps $x$ and $y$. From this we can compute the homotopy groups of $R$
  \[ \pi_sR \cong \begin{cases} k[x,y] & s=0 \\ (x,y)/(x,y)^2 & s=-1 \\ 0 & \text{otherwise} \end{cases}. \]
  $\F_p[x,y]$ is regular, but $(x,y)/(x,y)^2$ has tor dimension $2$ over $\F_p[x,y]$, so $R$ does not satisfy the conditions of \Cref{thm:opener}.

  Now we proceed to show that $R$ is regular and the connective cover map induces an equivalence of nc motives.
  In working with perfect $R$-modules we identify this category with the thick subcategory of
  $\Mod(\F_p[x,y])^{\omega}$ generated by $(x,y)$.  
  To see that $R$ is regular, first observe that $\F_p$ is connective in the standard $t$-structure for $R$ since it is a retract of $(x,y)\otimes_{\F_p[x,y]}\F_p$. From the extension $(x,y) \to \F_p[x,y] \to \F_p$ we can then conclude that $\F_p[x,y]$ is connective in this $t$-structure as well. This then implies that $\Mod(R)^{\omega}_{\geq0}$ is equivalent to $\Mod(\F_p[x,y])^{\omega}_{\geq0}$.
  Since $(x,y)$ represents the class $1$ in $K_0(\F_p[x,y])$, base-change along the map $\pi_0R = \F_p[x,y] \to R$, which can be identified with the functor $\otimes_{\F_p[x,y]}(x,y)$, induces multiplication by $1$ on the nc motive of $\F_p[x,y]$.
  \tqed
\end{exm}
  

In this example the natural map $K(\pi_0R) \to K(R)$ is an equivalence despite the fact that $R$ doesn't satisfy the conditions of \Cref{thm:opener}. The essential issue here is that $(x,y)$ is not flat, i.e. the base-change functor $\pi_0R \to R$ is not $t$-exact. Since, at its core \Cref{thm:opener} operates using Quillen's devissage theorem it cannot be used for examples of this type.
The equivalence of $K$-theories in this example arises because $R$ is Morita equivalent to its connective cover, which is a different (and less interesting) reason for them to agree.
As noted above, this implies that the connective cover map induces an equivalence of nc motives in this case, something which rarely happens for rings to which one can apply \Cref{thm:opener}. 

Another point contrasting with \Cref{thm:opener} is the fact that the equivalence $K(\pi_0R) \cong K(R)$ is not visible at the level of the homotopy ring of $R$. Indeed, if $R'$ is the trivial square zero extension of $\F_p[x,y]$ by $\Sigma^{-1}(x,y)/(x,y)^2$, then its homotopy ring agrees with $R$, but $K(\pi_0 R') \to K(R')$ is not an equivalence, because the map $\pi_0R' \to R'$ has the map $\F_p[x] \to \F_p[x]\oplus \Sigma^{-1}\F_p$ as a retract, which was shown in \Cref{exm:smalltorexm} to not be a $K$-theory equivalence.

\bibliographystyle{alpha}
\bibliography{ref}

\end{document}